\documentclass[12pt,a4paper]{amsart}

\usepackage{amsfonts, amsmath, amssymb, amsthm, amscd}
\usepackage{anysize}
\usepackage{graphicx}

\usepackage{color}

\marginsize{1.6cm}{1.6cm}{1.6cm}{1.6cm}

\newtheorem{theorem}{Theorem}

\newtheorem*{claim*}{Claim}

\theoremstyle{definition}

\theoremstyle{remark}

\renewcommand{\deg}{deg}

\newcounter{fig}
\newcommand{\f}{\refstepcounter{fig} Fig. \arabic{fig}. }

\title{On the number of non-hexagons in a planar tiling}
\author{Arseniy~Akopyan}
\address{Arseniy~Akopyan, Institute of Science and Technology Austria (IST Austria), Am Campus~1, 3400 Klosterneuburg, Austria}
\email{akopjan@gmail.com}

\begin{document}


\begin{abstract}
We give a simple proof of T.~Stehling's result~\cite{stehling1989uber}, that in any normal tiling of the plane with convex polygons with number of sides not less than six, all tiles except the finite number are hexagons.
\end{abstract}

\maketitle

A tiling of the plane with convex polygons is \emph{normal} if these polygons are not very thin (each contains a disk of radius $r$) and not very long (each is contained in a disc of radius~$R$).
It is well known that, in a normal tiling of the plane into convex polygons, the (upper) average side number of the polygons is at most $6$, cf. \cite[p. 61, inequality $\overline{p}\leq 6$]{toth1953lagerungen} and \cite{niven1978convex}. The proof is based on double counting the angle sums.
T. Stehling in his thesis \cite{stehling1989uber} studied the properties of combinatorics of normal tiling, where in particular, he proved the following theorem.

\begin{theorem}[T. Stehling~\cite{stehling1989uber}]
	\label{thm:main theorem}
	If the plane is tiled with convex polygons, such that all of them have at least six sides, contain a disk of radius $r$ and are contained in a disk of radius $R$, then all but a finite number of them are hexagons.
\end{theorem}

In other words, we can not ``spoil'' the honeycomb tiling with infinite number of polygons having more than six sides.
It seems that this result was missed by the mathematical community and the author learned it as an open question from S.~Markelov~\cite{markelov2014lj}.
We give a simple proof of it, expressing the upper bound in terms of the area and the diameter of the tiles.
Define the \emph{index} of a polygon as the number of its sides minus $6$.
Theorem~\ref{thm:main theorem} is a trivial corollary of the following theorem.

\begin{theorem}
	\label{thm:bound theorem}
	If the plane is tiled with convex polygons, which have at least six sides, have area at least $A$ and diameter not greater than $D$, then the sum of indices of the tiles is not greater than $\frac{2\pi D^2}{A}-6$.
\end{theorem}

\begin{proof}
	First, without the loss of generality we may assume that for each vertex there are exactly three tiles meeting at it.
	Indeed, if it is not so, we can think of this vertex as of several vertices connected by edges of zero length.
	We can also assume, that the tiling is edge-to-edge, that is any two adjacent tiles share a common vertex or a full edge.
	Indeed, if the interior of a side of a tile $P$ contains a vertex of another tile then we can consider this vertex as a vertex of $P$.
	These operations only increase the sum of indices.
	
	The proof is based on comparing the geometric upper bound on number of pieces in a metric disk of radius $i$ with the combinatorial lower bound.
	
	Lets start with the latter bound.
	Consider the dual triangulation $G=(V, E, T)$, where vertices $V$ correspond to tiles, and two vertices are connected by an edge if they share a common side. Three vertices form a triangle if three corresponding tiles have a common vertex.

	Supply the $1$-skeleton of $G$ by the natural metric, that is each edge has length one.
	We fix some vertex $o$ as the origin.

	Now we describe the combinatorics of the universal cover $\widetilde G$, which by simply-connectness of $G$ will coincide with it.
	Denote by $B(i)$ and $S(i)$ the disk and the circle in the $1$-skeleton of $\widetilde G$, having radius $i$ and centered at $o$.
	The number $\deg(v)-6$ is called the \emph{index} of a vertex~$v$.

	If the vertex $o$ has degree $k$, the ball $B(1)$ contains $k+1$ points and has perimeter $k$.
	In the disk $B(1)$ all of these $k$ vertices have degree $3$, two of the edges at every vertex belong to the circle $S(1)$.
	
	In general, the vertices of $S(i)$, in the graph restricted to $B(i)$, have degree $4$ or $3$.
	This will be clear from our construction. 
	Call them \emph{regular} and \emph{generating} vertices, respectively. 
	
	Let us describe the annulus between $B(i)$ and $B(i+1)$.
	Each edge of $S(i)$ has a triangle attached to it.
	If a vertex $v_j$ is regular in $B(i)$ and has index zero, then the third vertices of triangles attached to the edges $v_{j-1}v_j$ and $v_jv_{j+1}$ are two consecutive regular vertices of $S(i+1)$.
	If the vertex $v_j$ is not regular or/and its index is positive, then there are some additional edges leading to consecutive vertices of $S(i+1)$ between those two triangles.
	Since we describe the universal cover of $G$ we assume that all these vertices on $S(i+1)$ are new and different.
	(Although in the original $G$ they could coincide.)
	From this construction the vertices of $S_{i+1}$ have degree $3$ or $4$ in $B_{i+1}$. 
	The number of the above additional edges equals the index of $v_j$ plus one if $v_j$ is generating.
	All vertices of $S(i+1)$ obtained on these additional edges are generating, and  the third vertices of the triangles attached to $v_{j-1}v_j$ and $v_jv_{j+1}$ are regular.

	Since $G$ itself is simply-connected, it actually must coincide with $\widetilde G$.
	Hence, our description of the $B(i)$ is the description of metric balls in the $1$-skeleton of $G$.
		
	If $I(i)$ is the sum of indices of points of the circle $S(i)$ and $J(i)$ is the number of generating points on it, then $|S(i+1)|=|S(i)|+I(i)+J(i)$ and $J(i+1)=I(i)+J(i)$.
	
	Therefore, if the sum of indices of vertices of $G$ is at least $C$, then for sufficiently large $i$ we have $|J(i)|\geq C+6$.
	After summation we obtain $|S(i)|\geq (C+6)i-O(1)$.
	After another summation we have that the total number of points of $B(i)$ is at least $\frac{6+C}{2} i^2-O(i)$.
	
	Let us compare it with the similar geometric bound.
	Since the diameter of each tile is not greater than $D$, all tiles corresponding to vertices in $B(i)$ lie inside the disk of radius $i\cdot D +D$ (in the original metric of the plane) centered at a point of the tile corresponding to the origin $o$ of the graph $G$.
	Each tile has area at least $A$, so we can conclude that the number of tiles lying inside the disk is not greater than $\frac{\pi(i+1)^2D^2}{A}$ for any $i$. 
	That gives us the bound $C\leq \frac{2\pi D^2}{A}-6$.
\end{proof}

\parbox[b]{0.4\textwidth}{
	\begin{center}
		\includegraphics{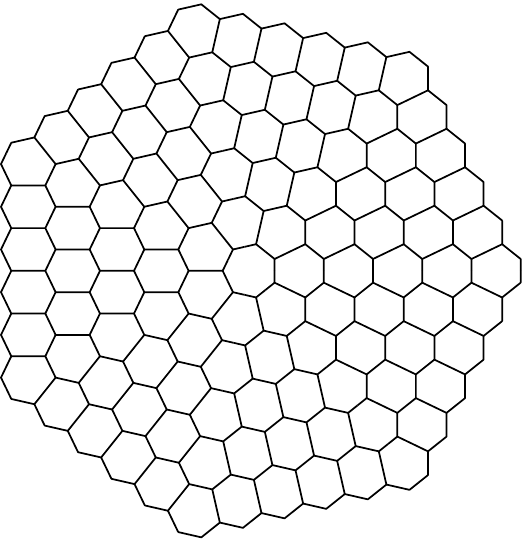}\\
	\end{center}
}
\hskip 0.3cm
\parbox[b]{0.5\textwidth}{
	\begin{center}
		\includegraphics{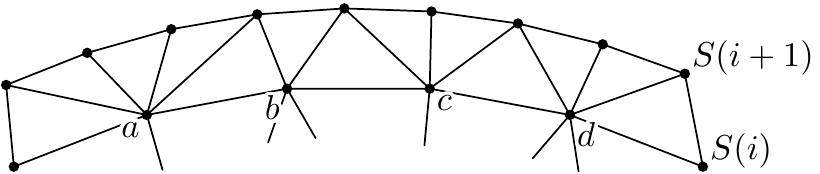}\\
	\end{center}
}

\parbox[t]{0.4\textwidth}{
	\f Tiling of the plane with hexagons and a single heptagon.
	}
\hskip 0.3cm
\parbox[t]{0.55\textwidth}{
	\f A part of the annulus between $S(i)$ and $S(i+1)$.
		Vertices $a$ and $c$ are generating, $a$ and $d$ have non-zero~index.
	}

\end{document}